\documentclass[12pt,psfig,reqno]{amsart}
\usepackage{mathrsfs}
\usepackage{txfonts}
\usepackage{amsmath}
\usepackage{bbm}
\usepackage{}
\usepackage{amssymb}
\usepackage{amsfonts}
\usepackage{amscd}
\usepackage{amssymb,amsfonts,latexsym}
\usepackage{graphics,verbatim}
\usepackage{graphicx}
\usepackage[usenames]{color} %\color{Red}
\makeatletter
\renewcommand\section{\@startsection {section}{1}{\z@}%
                                   {-3.5ex \@plus -1ex \@minus -.2ex}%
                                   {2.3ex \@plus.2ex}%
                                   {\centering\normalfont\bf}}

 \numberwithin{equation}{section}
\numberwithin{equation}{section}

\numberwithin{equation}{section}

\theoremstyle{plain}
\newtheorem{thm}{Theorem}[section]
\newtheorem{lemma}[thm]{Lemma}
\newtheorem{pro}[thm]{Proposition}
\newtheorem{cor}[thm]{Corollary}
\newtheorem{ex}[thm]{Example}
\newtheorem{de}[thm]{Definition}
\newtheorem{re}[thm]{Remark}
\newtheorem*{thm*}{Theorem}

\begin{document}
\title{Non-spectral problem for the planar self-affine measures}
\author{Jing-Cheng Liu, Xin-Han Dong$^*$ and Jian-Lin Li}
\address{Key Laboratory of High Performance Computing and Stochastic Information Processing
(Ministry of Education of China), College of Mathematics and
Computer Science, Hunan Normal University, Changsha, Hunan 410081,
China} \email{jcliu@hunnu.edu.cn} \email{xhdong@hunnu.edu.cn}

\address{School of Mathematics and Information Science, Shaanxi Normal University, Xi'an 710119, China}

\email{jllimath10@snnu.edu.cn}

\date{\today}
\keywords { Orthonormal set,  Spectral measure, Fourier transform, Zeros.}
\subjclass[2010]{Primary 28A80; Secondary 42C05, 46C05.}
\thanks{ The research is supported in part by the NNSF of China (Nos. 11571099, 11301175, 11571104 and 11571214),
the SRFDP of Higher Education (No. 20134306110003), the SRF of Hunan Provincial Education Department (No.14K057),
and the program for excellent talents in Hunan Normal University (No.ET14101).\\
$^*$Corresponding author.}

\begin{abstract}
In this paper, we consider the non-spectral problem for the planar
self-affine measures $\mu_{M,D}$ generated by an expanding integer
matrix $M\in M_2(\mathbb{Z})$ and  a finite digit set
$D\subset\mathbb{Z}^2$. Let $p\geq2$ be a positive integer,
$E_p^2:=\frac{1}{p}\{(i,j)^t:0\leq i,j\leq p-1\}$ and
$\mathcal{Z}_{D}^2:=\{x\in[0, 1)^2:\sum_{d\in D}{e^{2\pi i\langle
d,x\rangle}}=0\}$. We show that if
$\emptyset\neq\mathcal{Z}_{D}^2\subset E_p^2\setminus\{0\}$ and
$\gcd(\det(M),p)=1$, then there exist at most $p^2$ mutually
orthogonal exponential functions in $L^2(\mu_{M,D})$. In
particular, if $p$ is a prime, then the number $p^2$ is the best.
\end{abstract}

\maketitle

\section{\bf Introduction\label{sect.1}}

Let $M\in M_n(\mathbb{R})$ be an $n\times n$ expanding real matrix (that is,
all the eigenvalues of $M$ have moduli $>1$), and $D\subset\mathbb{R}^n$ be a finite subset with cardinality $\#(D)$.
Let $\{\phi_d(x)\}_{d\in D}$ be an iterated function system (IFS) defined by
\begin{equation*}\label {eq(1.1)}
\phi_d(x)=M^{-1}(x+d)   \ \  (x\in \mathbb{R}^n,  \ \ d\in D).
\end{equation*}
Then the IFS arises a natural \emph{self-affine measure} $\mu$ satisfying
\begin{equation}\label {eq(1.1)}
\mu=\mu_{M,D}=\frac{1}{\#(D)}\sum_{d\in D}\mu\circ\phi_d^{-1}.
\end{equation}
Such a measure $\mu_{M,D}$ is supported on the attractor $T(M,D)$ of the IFS $\{\phi_d\}_{d\in D}$ \cite{Hutchinson_1981}.

For a countable subset $\Lambda\subset\mathbb{R}^n$, let $\mathcal
{E}_\Lambda=\{e^{2\pi
i\langle\lambda,x\rangle}:\lambda\in\Lambda\}$. We call $\mu$  a
spectral measure, and $\Lambda$ a  spectrum of $\mu$ if $\mathcal
{E}_\Lambda$ is an orthogonal basis for $L^2(\mu)$. We also say
that $(\mu,\Lambda)$ is a $spectral$ $pair$. The existence of a
spectrum for $\mu$ is a basic problem in harmonic analysis, it was
initiated by Fuglede in his seminal paper \cite{Fuglede_1974}. The
first example of a singular, non-atomic, spectral measure was
given by Jorgensen and Pedersen in \cite{Jorgenson-Pederson_1985}.
This surprising discovery has received a lot of attention,  and
the research on the spectrality of self-affine measures has become
an interesting topic. Also, new spectral measures were found in
\cite{Hu-Lau_2008},
\cite{Dai_2012}-\cite{Dutkay-Haussermann-Lai_2015},
\cite{Li_2011}-\cite{Li_2012} and references cited therein. A
related problem is the non-spectral problem of  self-affine
measure. In \cite{Dutkay-Jorgensen}, Dutkay and Jorgensen showed
that if $ M=\left[ {\begin{array}{*{20}{c}}
{{p}}&{{0}}\\
{{0}}&{{{p}}}
\end{array}} \right]
$ with $p\in \mathbb{Z}\setminus 3\mathbb{Z}$, $p\geq 2$ and
\begin{align}\label{eq(1.2)}
\mathscr{D}=\left\{ {\left( {\begin{array}{*{20}{c}}
0\\
0
\end{array}} \right),\left( {\begin{array}{*{20}{c}}
1\\
0
\end{array}} \right),\left( {\begin{array}{*{20}{c}}
0\\
1
\end{array}} \right)} \right\},
\end{align}
then there are no 4 mutually orthogonal exponential functions in $L^2(\mu_{M,\mathscr{D}})$; they also prove that if
$
M=\left[ {\begin{array}{*{20}{c}}
{{2}}&{{1}}\\
{{0}}&{{{2}}}
\end{array}} \right],
$
then there exist at most $7$ mutually orthogonal exponential  functions in $L^2(\mu_{M,\mathscr{D}})$.
In \cite{Li_2008}, the third author of this paper proved that if  the expanding integer matrix
$
M=\left[ {\begin{array}{*{20}{c}}
{{a}}&{{b}}\\
{{0}}&{{{c}}}
\end{array}} \right]
$ with $ac\not\in 3\mathbb{Z}$, then there exist at most $3$ mutually orthogonal exponential  functions in $L^2(\mu_{M,\mathscr{D}})$,
and the number 3 is the best.  The third author also obtained the same conclusions for the expanding integer matrix
$
M=\left[ {\begin{array}{*{20}{c}}
{{a}}&{{b}}\\
{{d}}&{{{c}}}
\end{array}} \right]
$ with $\det(M)=ac-bd\not\in 3\mathbb{Z}$ in  \cite{Li_2009}.
In this paper, we will give a more general result which is suitable for more self-affine measures.
Before the statement of the main results, we first give some definitions and notations.

\par
For a positive integer $p\geq 2$ and a finite digit set $D\subset\mathbb{Z}^n$, let
\begin{equation}\label{eq(1.3)}
E_p^n:=\frac{1}{p}\{(l_1,l_2,\cdots,l_n)^t:0\leq l_1,\cdots,l_n\leq p-1\}, \ \ \ \ \ \ \ \mathring{E}_p^n:=E_p^n\setminus\{0\}
\end{equation}
and
\begin{equation}\label{eq(1.4)}
m_{D}(x)=\frac{1}{\#(D)}\sum\limits_{d\in D}{e^{2\pi i\langle d,x\rangle}}, \quad x\in\mathbb{R}^n,
\end{equation}
where $m_{D}(x)$ is called the\emph{ mask polynomial} of $D$ as
usual.

\par
Define $\mathcal {Z}(m_{D}):=\{ x\in \Bbb R^n:\;m_{D}(x)=0 \}$ and
\begin{equation}\label{eq(1.5)}
\mathcal{Z}_{D}^n:=\mathcal{Z}(m_{D})\cap [0, 1)^n.
\end{equation}
It is easy to see that $m_D$ is a $\mathbb{Z}^n$-periodic function for $D\subset \mathbb{Z}^n$. In this case,
\begin{equation}\label{eq(1.6)}
\mathcal{ Z}(m_D)=\mathcal{Z}_{D}^n+\mathbb{Z}^n.
\end{equation}

\begin{thm} \label{th(1.1)}
Let $p\geq 2$ be a positive integer,  $D\subset\mathbb{Z}^2$ be a
finite digit set, $M\in M_2(\Bbb Z)$ be an expanding integer
matrix, and let $\mu_{M,D}$, $\mathring{E}_p^2$,
$\mathcal{Z}_{D}^2$ be defined by \eqref{eq(1.1)}, \eqref{eq(1.3)}
and \eqref{eq(1.5)}, respectively. If
$\emptyset\neq\mathcal{Z}_{D}^2\subset\mathring{E}_p^2$ and
$\gcd(\det(M),p)=1$, then there exist at most $p^2$ mutually
orthogonal exponential functions  in $L^2(\mu_{M,D})$. In
particular, if $p$ is a prime, then the number $p^2$ is the best.
\end{thm}

 In fact, we can prove a more general result. In order to  state this conclusion, we need the following
definition.

\medskip

\begin{de}\label{de(1.2)}
Let  $\mu$ be a Borel probability measure with compact support on
$\mathbb{R}^n$.  Let $\Lambda$ be a finite or countable subset of
$\mathbb{R}^n$, and let ${\mathcal E}_{\Lambda}=\{e^{2\pi
i\langle\lambda,x\rangle}:\lambda\in\Lambda\}$. We denote
${\mathcal E}_{\Lambda}$ by  ${\mathcal E}^*_{\Lambda}$ if
${\mathcal E}_{\Lambda}$ is a maximal orthogonal set of
exponential functions in $L^2(\mu)$. Let
\begin{align}\label{eq(1.7)}
n^*(\mu):=\sup\{ \#\Lambda :  \mathcal {E}_{\Lambda}^* \hbox { is a maximal orthogonal set} \},
\end{align}
 and call $n^*(\mu)$  the maximal cardinality of the  orthogonal exponential functions  in $L^2(\mu)$.
\end{de}

\medskip

Obviously, ${\mathcal E}^*_\Lambda$ is not unique,  even  among
those $\Lambda$ with equal cardinality. For a positive integer
$p\geq 2$, let $\mathfrak{M}_p$ denote the class of all expanding
matrices $M\in M_2(\Bbb Z)$ with  $\gcd(\det(M),p)=1$. If
$D\subset\mathbb{Z}^2$ is a finite digit set with
$\emptyset\neq\mathcal{Z}_{D}^2\subset\mathring{E}_p^2$, then
$n^*(\mu_{M,D})\leq p^2$ for $M\in \mathfrak{M}_p$\; by Theorem
\ref{th(1.1)}. Furthermore, we can divide (see Theorem
\ref{th(1.3)} below) $\mathfrak{M}_p$  into two disjoint
subclasses $\mathfrak{M}_p^{(1)}$ and $\mathfrak{M}_p^{(2)}$ such
that $n^*(\mu_{M,D})\leq\widetilde{P}$ for $M\in
\mathfrak{M}_p^{(1)}$, and $n^*(\mu_{M,D})=p^2$ for $M\in
\mathfrak{M}_p^{(2)}$, where
\begin{align}\label{eq(1.8)}
\widetilde{P}=
\left\{
\begin{array}{ll}
\frac{p^2}{2},\quad &{\rm if}\;  p \; {\rm is \; even}, \\
\frac{p^2-1}{2},\quad &{\rm if}\;  p>3\;  {\rm is \; odd},\\
3,\quad &{\rm if}\;  p=3.
\end{array}
\right.
\end{align}
For $M\in \mathfrak{M}_p$,  we denote the transposed conjugate of
$M$ by $M^*$. Let $\lambda\in \mathring{E}_p^2$. We observe (see
Proposition \ref{th(2.2)}) that ${M^*}^j\lambda\in
\mathring{E}_p^2$ (mod $\mathbb{Z}^2$) for each $j$. From this
conclusion, we can easily prove
$\bigcup_{j=1}^\infty\{{M^*}^j\lambda\}=\bigcup_{j=1}^{p^2-1}\{{M^*}^j\lambda\}$
(mod $\mathbb{Z}^2$), since there exist only $p^2-1$ elements in
$\mathring{E}_p^2$. Let
\begin{align}\label{eq(1.9)}
\left\{
\begin{array}{ll}
\mathfrak{M}_p^{(1)}=\{M:\;\bigcup_{j=1}^{p^2-1}{M^*}^j\mathcal{Z}_{D}^2\subsetneqq \mathring{E}_p^2 ({\rm{mod}}  \ \mathbb{Z}^2)\},\\
\mathfrak{M}_p^{(2)}=\{M:\;\bigcup_{j=1}^{p^2-1}{M^*}^j\mathcal{Z}_{D}^2=\mathring{E}_p^2({\rm{mod}} \  \mathbb{Z}^2)\}.
\end{array}
\right.
\end{align}
It is possible that $\mathfrak{M}_p^{(1)}=\emptyset$ (see Example \ref{th(4.2)}), but if $p$ is a prime,
then $\mathfrak{M}_p^{(2)}\neq \emptyset$ (see Proposition \ref{th(2.6)}).

\begin{thm} \label{th(1.3)}
 Under the  conditions of Theorem \ref{th(1.1)}, let $n^*(\mu_{M,D})$ be given by \eqref{eq(1.7)}. Then
\begin{align*}
n^*(\mu_{M,D})
\left\{
\begin{array}{ll}
\leq \widetilde{P},\quad &{\rm if}\;  M\in \mathfrak{M}_p^{(1)}, \\
=p^2,\quad &{\rm if}\;  M\in \mathfrak{M}_p^{(2)} ,
\end{array}
\right.
\end{align*}
where   $\widetilde{P}$ and $\mathfrak{M}_p^{(1)}$, $\mathfrak{M}_p^{(2)}$  are defined by \eqref{eq(1.8)} and \eqref{eq(1.9)},
respectively.
\end{thm}

By Theorem \ref{th(1.3)}, to prove  Theorem \ref{th(1.1)}, we only
need to prove that if $p$ is a prime, then
$\mathfrak{M}_p^{(2)}\neq \emptyset$.  In particular, if $p=2$ or
$3$, we can  get a more stronger  result.

\begin{thm} \label{th(1.4)}
Under the conditions of Theorem \ref{th(1.3)}, if $p=2$ or $3$, then
\begin{align*}
n^*(\mu_{M,D})=
\left\{
\begin{array}{ll}
p,\quad &{\rm if}\; M\in \mathfrak{M}_p^{(1)}, \\
p^2,\quad &{\rm if}\;  M\in \mathfrak{M}_p^{(2)} .
\end{array}
\right.
\end{align*}
\end{thm}

\medskip

The organization of the paper is as follows. We give several
preparatory definitions and  conclusions in Section 2, and prove
Theorems \ref{th(1.1)}, \ref{th(1.3)} and \ref{th(1.4)} in Section
3. As an application, in Section 4 we give a complete discussion
for the Sierpinski measure in \cite{Li_2009} and construct an
example to illustrate some special case.

\section{\bf Preliminaries\label{sect.2}}

In this section, we give some  preliminary definitions and
propositions. We will start with an introduction to the Fourier
transform. Let $M\in M_n(\mathbb{R})$ be an $n\times n$ expanding
real matrix, $D\subset\mathbb{R}^n$ be a finite subset with
cardinality $|D|$. Let $\mu_{M,D}$ be defined by \eqref{eq(1.1)}.
In the study of spectrality of $\mu_{M,D}$, the Fourier transform
$\hat{\mu}_{M,D}(\xi)=\int e^{2\pi i\langle x,\xi\rangle
}d\mu_{M,D}(x), (\xi\in\mathbb{R}^n)$ of $\mu_{M,D}$ plays an
important role. It follows from \cite{Dutkay-Jorgensen} that
\begin{equation}\label {eq(2.1)}
\hat{\mu}_{M,D}(\xi)=\prod_{j=1}^\infty m_D({M^{*}}^{-j}\xi), \ \ \ (\xi\in\mathbb{R}^n)
\end{equation}
where $M^*$ denotes the transposed  conjugate of $M$, and
$$
m_D(x)=\frac{1}{\#(D)}\sum\limits_{d\in D}{e^{2\pi i\langle d,x\rangle}},\;(x\in\mathbb{R}^n).
$$
 For any $\lambda_1, \lambda_2\in \mathbb{R}^n$, $\lambda_1\neq \lambda_2$, the orthogonality condition
$$
0=\langle e^{2\pi i \langle \lambda_1,x\rangle},e^{2\pi i \langle \lambda_2,x\rangle}\rangle_{L^2(\mu_{M,D})}
=\int e^{2\pi i \langle\lambda_1-\lambda_2,x\rangle}d\mu_{M,D}(x)=\hat{\mu}_{M,D}(\lambda_1-\lambda_2)
$$
relates to the zero set $\mathcal{Z}(\hat{\mu}_{M,D})$ directly.
It is easy to see that for a countable subset
$\Lambda\subset\mathbb{R}^n$,  $\mathcal {E}_\Lambda=\{e^{2\pi
i\langle\lambda,x\rangle}:\lambda\in\Lambda\}$ is an orthonormal
family of $L^2(\mu_{M,D})$ if and only if
 \begin{equation}\label {eq(2.2)}
 (\Lambda-\Lambda)\setminus\{0\}\subset\mathcal{ Z}(\hat{\mu}_{M,D}).
\end{equation}
From \eqref{eq(2.1)}, we have
$\mathcal{Z}(\hat{\mu}_{M,D})=\Big\{\xi\in \mathbb{R}^n: \exists\; j\in \mathbb{N} \ {\rm {such  \ that}}\ m_D(M^{*-j}\xi)=0\Big\}$.
Hence
\begin{equation}\label {eq(2.3)}
\mathcal{ Z}(\hat{\mu}_{M,D})=\bigcup_{j=1}^{\infty}M^{*j}(\mathcal{ Z}(m_D)),
\end{equation}
where $\mathcal{ Z}(m_D)=\{x\in\mathbb{R}^n:m_D(x)=0\}$.

The following lemma  can be used to determine the case that there
are only finite orthogonal exponential functions in
$L^2(\mu_{M,D})$.

\begin{lemma}\cite[Theorem 3.1]{Dutkay-Jorgensen}\label{th(2.1)}\;\; For a $n\times n$ expanding integer matrix $M$
and a finite digit set $D\subset \mathbb{Z}^n$, let $\mu_{M,D}$
and $\mathcal{Z}_D^n$ be  defined by \eqref{eq(1.1)} and
\eqref{eq(1.5)}, respectively. If $\mathcal{Z}_D^n$ is contained
in a set $\mathrm{Z}'\subset [0, 1)^n$ of finite cardinality
$\#(\mathrm{Z}')$, which does not contain $0$, and that satisfies
the property
$$
M^{*}(\mathrm{Z}'+\mathbb{Z}^n)\subseteq \mathrm{Z}'+\mathbb{Z}^n ,
$$
then there exist at most $\#(\mathrm{Z}')+1$ mutually orthogonal exponential functions in $L^2(\mu_{M,D})$.
In particular, $\mu_{M,D}$ is not a spectral measure.
\end{lemma}

\begin{pro} \label{th(2.2)}
Let $p\geq 2$ be a positive integer  and $\mathring{E}_p^2$ be
defined by \eqref{eq(1.3)}, and let $M\in M_2(\mathbb{Z})$ with
$\gcd(\det(M),p)=1$. Then $M(p \mathring{E}_p^2)=
p\mathring{E}_p^2$ $({\rm mod}\;p\Bbb Z^2)$, equivalently, $M(
\mathring{E}_p^2)= \mathring{E}_p^2$ $({\rm mod} \; \Bbb Z^2)$.
\end{pro}

\begin{proof} We first prove $M( p\mathring{E}_p^2)\subset p\mathring{E}_p^2$ $($mod $p\mathbb{Z}^2$$)$.
Suppose on the contrary that there exist  $(l_1,l_2)^t\in p
\mathring{E}_p^2$ and $v_0\in \mathbb{Z}^2$ such that
\begin{equation}\label{eq(2.4)}
M\left( {\begin{array}{*{20}{c}}
l_1\\
l_2
\end{array}} \right)=pv_0.
\end{equation}
Multiplying $\det(M)p^{-1}M^{-1}$ on both sides of   \eqref{eq(2.4)}, we get
\begin{equation*}
\frac{\det(M)}{p}\left( {\begin{array}{*{20}{c}}
l_1\\
l_2
\end{array}} \right)=\det(M)M^{-1}v_0.
\end{equation*}
Obviously, $\det(M)M^{-1}v_0\in\mathbb{Z}^2$, but the left side of the above equation can not be a integer vector
since $\gcd(\det(M),p)=1$ and $(l_1,l_2)^t\in p \mathring{E}_p^2$, this contradiction shows that
$M( p\mathring{E}_p^2)\subset p\mathring{E}_p^2$ $($mod $p\mathbb{Z}^2$$)$.

We now prove $M( p \mathring{E}_p^2)= p \mathring{E}_p^2\;({\rm mod} \;p\mathbb{Z}^2)$.
For any $\lambda\neq \lambda'\in p \mathring{E}_p^2$, it follows from $(\lambda-\lambda')\in p \mathring{E}_p^2\;({\rm mod} \; p
\mathbb{Z}^2)$ and $M( p \mathring{E}_p^2)\subset p \mathring{E}_p^2$ $($mod $p\mathbb{Z}^2$$)$ that $M\lambda\not=M\lambda'
\;({\rm mod} \;p\mathbb{Z}^2)$. Hence $\#(M( p\mathring{E}_p^2) \ ({\rm mod} \;p\mathbb{Z}^2))= \#(p \mathring{E}_p^2)$, this gives that
$M( p \mathring{E}_p^2)= p \mathring{E}_p^2\;({\rm mod} \;p \mathbb{Z}^2)$ since $M(p \mathring{E}_p^2)\subset p \mathring{E}_p^2\;({\rm mod} \;p \mathbb{Z}^2)$.
It also shows that
$M( \mathring{E}_p^2)= \mathring{E}_p^2$ $({\rm mod} \; \Bbb Z^2)$.
\end{proof}

To prove that the number $p^2$ is the best in Theorem \ref{th(1.1)}, we need the following knowledge related to the theory of number and algebra.

For a positive number $m$, let $\varphi(m)$ denote the \emph{Euler's phi function} ( also be called \emph{Euler's totient function}) which  equal to the number
of integers in the set $\{1,2,\cdots, m-1\}$ that are relatively prime to $m$.  For more information about the Euler's phi function,
the reader can refer to \cite{Nathanson_1996}.
The following lemma is the famous \emph{Euler's theorem}.

\begin{lemma}\cite[Theorem 2.12]{Nathanson_1996}\label{th(2.3)}\;\; Let $m$ be a positive integer, and let $a$ be an integer
relatively prime to $m$. Then $a^{\varphi(m)}=1({\rm mod} \ m).$
\end{lemma}

For a prime $p$, let $\mathbb{F}_p:=\mathbb{Z}/p\mathbb{Z}$ denote  the residue class fields.
All nonsingular $n\times n$ matrices over $\mathbb{F}_p$ form a finite group under matrix multiplication,
called the \emph{ general linear group} $GL(n,\mathbb{F}_p)$.

\begin{de}\label{de(2.4)}
Let $f(x)\in \mathbb{F}_p[x]$ be a nonzero polynomial. If $f(0)\neq0$, then
the least positive integer $q$ for which $f(x)$ divides $x^q-1$ is called the order of
$f$ and denoted by $ord(f)$.
\end{de}

\begin{de}\label{de(2.5)}
Let $M\in GL(n,\mathbb{F}_p)$, then the least positive integer $s$ for which $M^s=I$ is called the order of
$M$ and denoted by $O(M)$, where $I$ is the  identity matrix in $GL(n,\mathbb{F}_p)$.
\end{de}

Let $f(x)=x^n-a_{n-1}x^{n-1}-\cdots-a_1x-a_0 \in \mathbb{F}_p[x]$ and  $f(0)\neq0$, the companion matrix $M$ of $f(x)$ is defined by
\begin{align*}
        M:=\begin{bmatrix}
0&0&0&\cdots &0 &a_0\\
1&0&0&\cdots &0 &a_1\\
0&1&0&\cdots &0 &a_2\\
\vdots & \vdots & \vdots & \ddots & \vdots& \vdots \\
0&0&0&\cdots &1 &a_{n-1}\\
  \end{bmatrix}.
\end{align*}
It follows from Theorem 6.26 in \cite{Lidel-Niederreiter_1994} that $M\in GL(n,\mathbb{F}_p)$ and $ord(f)=O(M)$.
Further, for any positive integer $n\geq 1$,  if $f(x)\in \mathbb{F}_p[x]$ is a primitive polynomial of degree $n$,
then  by Theorem 3.16 in \cite{Lidel-Niederreiter_1994}, we know that $f$ is monic,  $f(0)\neq0$ , and $ord(f)=p^n-1$.
It is well known that there exist $\varphi(p^n-1)/n$ primitive polynomials with degree $n$ over $\mathbb{F}_p$(see $P_{87}$
Theorem 4.1.3 of \cite{Mullen-Panario_2013} ), where $\varphi$ is  Euler's phi function.

The following Proposition can be proved more easily by using the group and orbit theory, however,
we will prove it directly by avoiding to introduce  more definitions and notations.

\begin{pro}\label{th(2.6)}
Let $p$ be a prime, then there exists a matrix $M\in  M_2(\mathbb{Z})$ with $\gcd(\det(M),p)=1$ such that
\begin{align} \label{eq(2.5)}
\bigcup_{j=1}^{p^2-1}\{M^j\lambda\}=\mathring{E}_p^2\;({\rm mod}\;\mathbb{Z}^2)\;\;\;{\rm for\;all}\;\;\lambda\in \mathring{E}_p^2.
\end{align}
\end{pro}

\begin{proof}
 Let $f(x)=x^2-a_1x-a_0\in \mathbb{F}_p[x]$ be a primitive polynomial and
$ M=\begin{bmatrix}
0&a_0\\
1&a_1
\end{bmatrix}
$
be the companion matrix of $f(x)$, then $M\in GL(2,\mathbb{F}_p)$ and $ord(f)=O(M)=p^2-1$.
Moreover,  $\gcd(\det(M),p)=\gcd(-a_0,p)=1$ because
$f(0)\neq 0$ and $p$ is prime.

In order to prove the matrix $M$ satisfies \eqref{eq(2.5)}, we first need to prove
\begin{align}\label{eq(2.6)}
\Big\{M\frac{e_1}{p},M^2\frac{e_1}{p},\cdots,M^{p^2-1}\frac{e_1}{p}\Big\}=\mathring{E}_p^2\;({\rm mod} \;\mathbb{Z}^2),
\end{align}
where $e_1=(1,0)^t$. Assume that the \eqref{eq(2.6)} has been
proved. For any $\lambda\in \mathring{E}_p^2$, \eqref{eq(2.6)}
implies that there exist $k_0\in \{1,\cdots,p^2-1\}$ and $v_0\in
\mathbb{Z}^2$  such that $M^{k_0}\frac{e_1}{p}=\lambda+v_0$. By $\gcd(\det(M),p)=1$ and
Proposition \ref{th(2.2)}, we have
\begin{align*}
\mathring{E}_p^2&=M^{k_0}(\mathring{E}_p^2)=M^{k_0}(\{M\frac{e_1}{p},M^2\frac{e_1}{p},\cdots,M^{p^2-1}\frac{e_1}{p}\}) \;({\rm mod} \;\mathbb{Z}^2)\\
&=\{M(\lambda+v_0),M^{2}(\lambda+v_0),\cdots,M^{p^2-1}(\lambda+v_0)\} \;({\rm mod} \;\mathbb{Z}^2)\\
&=\{M\lambda,M^{2}\lambda,\cdots,M^{p^2-1}\lambda\}\;({\rm mod} \;\mathbb{Z}^2).
\end{align*}
This indicates that $M$ is exactly what we need.
\par
We now prove \eqref{eq(2.6)}. Let $e_2=(0,1)$, and write
$$
k_1=\min\{k: M^ke_1=e_1({\rm mod} \ p\mathbb{Z}^2)\}, \ \ \ \ \ \ \ \ \ \ k_2=\min\{k: M^ke_2=e_2({\rm mod} \ p\mathbb{Z}^2)\}.
$$
Obviously, $1\leq k_1,k_2\leq p^2-1$ as $O(M)=p^2-1$. We conclude
\begin{align}\label{eq(2.7)}
\{Me_1,\cdots,M^{k_1}e_1\}\;({\rm mod} \; p \mathbb{Z}^2)\;{\rm contains \; exact}\;k_1\;{\rm different\; elements}.
\end{align}
Suppose, on the contrary, that there exist $1\leq l<l'\leq k_1$ and $v_0\in\mathbb{Z}^2$
such that $M^{l'}e_1=M^le_1+pv_0$, i.e., $M^l(M^{l'-l}e_1-e_1)=pv_0.$
Noting that $(M^{l'-l}e_1-e_1)\in p E_p^2\;({\rm mod}\;p\mathbb{Z}^2)$, by  Proposition \ref{th(2.2)},
we have $M^{l-l'}e_1-e_1=0\;({\rm mod} \;p\mathbb{Z}^2)$, which contradicts with the definition of $k_1$. Hence \eqref{eq(2.7)} holds. Similarly, we have
\begin{align}\label{eq(2.8)}
\{Me_2,\cdots,M^{k_2}e_2\}\;({\rm mod} \; p \mathbb{Z}^2)\;{\rm contains \; exact}\;k_2\;{\rm different\; elements}.
\end{align}
We further claim that
\begin{align}\label{eq(2.9)}
k_1=k_2=p^2-1\;.
\end{align}
If $k_1<k_2$,  then by using $Me_1=e_2$ (a directly check), we have $M^{k_1+1}e_2=M e_2$, this contradicts with \eqref{eq(2.8)}. If $k_1>k_2$,
then  $M^{k_2+1}e_1=M e_1$, this contradicts with \eqref{eq(2.7)}. So the claim $k_1=k_2$ follows.
From the definitions of $k_1,k_2$, we have $M^{k_1}\{e_1,e_2\}=\{e_1,e_2\}$ (mod $p \mathbb{Z}^2$),
hence $M^{k_1}I=M^{k_1}=I$ (mod  $p M_2(\mathbb{Z})$) which shows \eqref{eq(2.9)} holds since $k_1\leq p^2-1$ and $O(M)=p^2-1$.

\par
It follows from Proposition \ref{th(2.2)} that $\{Me_1,M^2e_1,
\cdots,M^{p^2-1}e_1\}\subset p \mathring{E}_p^2 \;({\rm mod} \;p \mathbb{Z}^2)$. Since $\{Me_1,M^2e_1,\cdots,M^{p^2-1}e_1\}\;({\rm mod} \;p \mathbb{Z}^2)$
contains exact $p^2-1$ different elements by \eqref{eq(2.7)} and \eqref{eq(2.9)}, and since $\#(p \mathring{E}_p^2)=p^2-1$, we get
$\{Me_1,M^2e_1,\cdots,M^{p^2-1}e_1\}=p\mathring{E}_p^2\;({\rm mod} \;p \mathbb{Z}^2)$, and the required result \eqref{eq(2.6)} holds.
\par
This completes the proof of Proposition \ref{th(2.6)}.
\end{proof}

To prove the Theorems, we also need the following proposition.

\begin{pro}\label{th(2.7)}
For  a positive integer $p\geq 2$, let $E_p^2$ and $\widetilde{P}$ be defined by \eqref{eq(1.3)}
and \eqref{eq(1.8)}, respectively. Let $m=\widetilde{P}+1$ and $A=\{\lambda_1, \lambda_2,\cdots,\lambda_m\} \subset E_p^2$ with $\lambda_i\neq \lambda_j$ for any
$i\neq j$. Then $A-A=E_p^2$ $(${\rm mod} \ $\mathbb{Z}^2$$)$.
\end{pro}

\begin{proof}
 Obviously, $(A-A)\subseteq E_p^2$ (mod $\mathbb{Z}^2)$. If we suppose, on the contrary, that $(A-A)\subsetneqq E_p^2$ (mod $\mathbb{Z}^2)$,
then there exists $\tau\in \mathring{E}_p^2$ satisfies $\tau\not\in (A-A)$ (mod $\mathbb{Z}^2)$. We claim that $-\tau\not\in (A-A)$ (mod $\mathbb{Z}^2)$, if not,  there are $1\leq i_0\neq j_0\leq m$ and $v_0\in \mathbb{Z}^2$ such that
$-\tau=\lambda_{i_0}-\lambda_{j_0}+v_0$. Hence $\tau=\lambda_{j_0}-\lambda_{i_0}-v_0$, this implies that $\tau\in (A-A)$ (mod $\mathbb{Z}^2) $, which contradicts with $\tau\notin (A-A)$ (mod $\mathbb{Z}^2$). So the claim $-\tau\not\in (A-A)$ (mod $\mathbb{Z}^2)$ follows.
\par
Let $\beta_i=\lambda_i-\tau, \;\gamma_i=\lambda_i+\tau$ for $i=1,2,\cdots,m$. Then for each $i\in\{1,2,\cdots,m\}$, we have $\beta_i, \gamma_i \notin A$ (mod $\mathbb{Z}^2)$. Otherwise if $\beta_{i_0}\in A$ (mod $\mathbb{Z}^2)$, then $\beta_{i_0}-\lambda_{i_0}\in (A-A)$ (mod $\mathbb{Z}^2$), which contradicts with $\beta_{i_0}-\lambda_{i_0}=-\tau\not\in (A-A)$ (mod $\mathbb{Z}^2$). Similarly, $\gamma_{i_0}\in A$ (mod $\mathbb{Z}^2)$ also yields a contradiction. Thus we get
\begin{align}\label{eq(2.10)}
\{\beta_1,\;\cdots,\beta_m\}\cup\{\gamma_1,\cdots,\;\gamma_m\}\subset E_p^2\setminus A\;({\rm mod}\;\mathbb{Z}^2).
\end{align}
Noting that $\beta_i-\beta_j=\gamma_i-\gamma_j=\lambda_i-\lambda_j\not=0$ (mod $\mathbb{Z}^2$) for $i\neq j$, we have
\begin{align}\label{eq(2.11)}
\beta_i\not=\beta_j\; ({\rm mod} \;\mathbb{Z}^2) \quad {\rm and } \quad  \gamma_i\neq \gamma_j\; ({\rm mod} \;\mathbb{Z}^2)\quad {\rm if } \quad  i\not=j\;.
\end{align}
\par
(i)  \ If $p\neq 3 $, by the definition of $\widetilde{P}$ in \eqref{eq(1.8)}, we have $m=\frac{p^2}{2}+1$ when $p$ is even and $m=\frac{p^2-1}{2}+1$ when $p$ is odd. By \eqref{eq(2.10)}-\eqref{eq(2.11)}, we have $m=\#\Big\{\{\beta_1,\;\cdots,\beta_m\}\;({\rm mod} \;\mathbb{Z}^2)\Big\}\leq \#(E_p^2\setminus A)=p^2-m$, i.e., $2m\leq p^2$, this is impossible. Hence $(A-A)=E^2_p$(mod $\mathbb{Z}^2)$.
\par
(ii) \ If $p=3$, then $m=4$. By noting that $\#(E_3^2\setminus A)=5$ and \eqref{eq(2.10)}-\eqref{eq(2.11)}, there  exist at least one pair of $\beta_{i_1}$ and $\gamma_{j_1}$  such that $\beta_{i_1}=\gamma_{j_1}$(mod $\mathbb{Z}^2$), i.e., $\lambda_{i_1}-\tau=\lambda_{j_1}+\tau+v$, or $\lambda_{i_1}-\lambda_{j_1}-v=2\tau$  for some $v\in\mathbb{Z}^2$. This is equivalent to $\lambda_{i_1}-\lambda_{j_1}=-\tau$ (mod $\mathbb{Z}^2$) since $3\tau\in \mathbb{Z}^2$, it contradicts with $-\tau \notin (A-A)$ (mod $\mathbb{Z}^2$). This shows that $(A-A)=E^2_3$(mod $\mathbb{Z}^2)$.
\end{proof}

\section{\bf The proofs of the main theorems \label{sect.3}}

In this section, we first  prove  Theorem \ref{th(1.3)} by
Propositions \ref{th(2.2)} and \ref{th(2.7)}, and then prove
Theorem \ref{th(1.1)} by applying Proposition \ref{th(2.6)}.
Finally we prove   Theorem \ref{th(1.4)}.

\begin{proof}[The proof of Theorem \ref{th(1.3)}]
(i) We first prove $n^*(\mu_{M,D})\leq \widetilde{P}$ for $M\in
\mathfrak{M}_p^{(1)}$,  where $\widetilde{P}$ is given by
\eqref{eq(1.8)}. We will prove this by contradiction. Before the
proof, we first give some properties on the zeros of Fourier
transform $\hat{\mu}_{M,D}$.

For any $\lambda\in\mathcal{Z}_{D}^2\subset \mathring{E}_p^2$, by Proposition \ref{th(2.2)}, we see that $M^{*j}\lambda\in\mathring{E}_p^2$
(mod $\mathbb{Z}^2)$ for any $j$.  Note that $\mathring{E}_p^2$ has only $p^2-1$ different elements.
Then $\bigcup_{j=1}^{\infty}M^{*j}\lambda=\bigcup_{j=1}^{p^2-1}M^{*j}\lambda$ (mod $\mathbb{Z}^2)$, and so
\begin{align}\label{eq(3.1)}
\bigcup_{j=1}^{\infty}M^{*j}(\mathcal{Z}_{D}^2)=\bigcup_{j=1}^{p^2-1}M^{*j}(\mathcal{Z}_{D}^2) \ ({\rm{mod} }  \ \mathbb{Z}^2).
\end{align}
Since $M\in \mathfrak{M}_p^{(1)}$ is an integer matrix, it follows from  \eqref{eq(2.3)}, \eqref{eq(1.6)},  \eqref{eq(3.1)} and \eqref{eq(1.9)} that
\begin{align}\label{eq(3.2)}\nonumber
\mathcal{ Z}(\hat{\mu}_{M,D})&=\bigcup_{j=1}^{\infty}M^{*j}(\mathcal{ Z}(m_D))=\bigcup_{j=1}^{\infty}M^{*j}(\mathcal{Z}_{D}^2+\mathbb{Z}^2)\\
&\subset\left(\bigcup_{j=1}^{\infty}M^{*j}(\mathcal{Z}_{D}^2)\right)+\mathbb{Z}^2= \left(\bigcup_{j=1}^{p^2-1}M^{*j}
(\mathcal{Z}_{D}^2)\right) +\mathbb{Z}^2\subset \mathring{E}_p^2+\mathbb{Z}^2.
\end{align}

Suppose on the contrary that there exists $\Lambda'=\{0, \lambda_1, \lambda_2, \cdots, \lambda_{\widetilde{P}}\}\subset \Bbb R^2$ such that
$\mathcal {E}_{\Lambda'}=\{e^{2\pi i\langle\lambda,x\rangle}:\lambda\in \Lambda'\}$ is a orthogonal set in $L^2(\mu_{M,D})$.
The orthogonal condition implies that
$$
(\Lambda'-\Lambda')\setminus\{0\}\subset \mathcal{Z}(\widehat{\mu}_{M,D}),
$$
and hence $\lambda_1-0,\cdots, \lambda_{\widetilde{P}}-0 \in \mathring{E}_p^2$ (mod $\mathbb{Z}^2$) by \eqref{eq(3.2)}.
Because $\mathcal{Z}(\widehat{\mu}_{M,D})$ does not contain any integer, we have $\lambda_j-\lambda_k\notin \mathbb{Z}^2$,
i.e. $\lambda_j\neq\lambda_k$ (mod $\mathbb{Z}^2$) for $j\neq k$. This yields that $A:=\Lambda'$(mod $\mathbb{Z}^2$) has $\widetilde{P}+1$
different elements in $E^2_p$. By Proposition \ref{th(2.7)}, we have $A-A=E^2_p$ (mod $\mathbb{Z}^2)$, and hence
$$
\mathring{E}_p^2\subset \mathcal{Z}(\widehat{\mu}_{M,D}) \ {\rm ( mod}\;\mathbb{Z}^2{\rm )}.
$$
On the other hand, since
$\bigcup_{j=1}^{p^2-1}{M^*}^j\mathcal{Z}_{D}^2\subsetneqq
\mathring{E}_p^2$ for $M\in \mathfrak{M}_p^{(1)}$,
\eqref{eq(2.3)}, \eqref{eq(1.6)} and \eqref{eq(3.1)} yield that
$$
\mathcal{Z}(\widehat{\mu}_{M,D})=\bigcup_{j=1}^{\infty}M^{*j}(\mathcal{ Z}(m_D))=\bigcup_{j=1}^{\infty}M^{*j}(\mathcal{Z}_{D}^2+\mathbb{Z}^2)
=\bigcup_{j=1}^{p^2-1}{M^*}^j\mathcal{Z}_{D}^2\subsetneqq \mathring{E}_p^2 \ {\rm (mod}\;\mathbb{Z}^2{\rm )}.
$$
This contradiction shows that such $\mathcal {E}_{\Lambda'}$ does not exist. Hence there exist at most $\widetilde{P}$ mutually
orthogonal exponential functions in $L^2(\mu_{M,D})$, i.e.,  $n^*(\mu_{M,D})\leq \widetilde{P}$.

(ii) We now prove that $n^*(\mu_{M,D})=p^2$ for $M\in \mathfrak{M}_p^{(2)}$.

Firstly, we prove that $n^*(\mu_{M,D})\leq p^2$. Since $M$ is an integer matrix and $\gcd(\det(M),p)=1$, by Proposition \ref{th(2.2)}, we have
$M(\mathring{E}_p^2+\mathbb{Z}^2)= M (\mathring{E}_p^2)+M(\mathbb{Z}^2)\subset M (\mathring{E}_p^2)+\mathbb{Z}^2=\mathring{E}_p^2+\mathbb{Z}^2$.
It follows from  Proposition \ref{th(2.1)} and $\mathcal{Z}_{D}^2\subset \mathring{E}_p^2$ that
\begin{align}\label{eq(3.3)}
n^*(\mu_{M,D})\leq \#(\mathring{E}_p^2)+1=p^2.
\end{align}

Secondly, we prove that $n^*(\mu_{M,D})\geq p^2$.  We will prove it by finding out exact $p^2$ mutually orthogonal exponential functions in $L^2(\mu_{M,D})$.

Let $\det(M)=L$ and $\varphi(p)$ denote the Euler phi function.  It follows from  $\gcd(L,p)=1$  and Lemma \ref{th(2.3)} that there exist integer $n$ such that
 \begin{align}\label{eq(3.4)}
 L^{\varphi(p)}=np+1.
 \end{align}
 For any $\lambda\in \mathring{E}_p^2$, we first prove that
\begin{align}\label{eq(3.5)}
{M^*}^j\left(\lambda+\mathbb{Z}^2\right)\supset L^{\varphi(p)j} \left({M^*}^j\lambda+\mathbb{Z}^2\right),
\end{align}
which is equivalent to
\begin{align*}
\lambda+\mathbb{Z}^2\supset  L^{\varphi(p)j}M^{*^{-j}} \left({M^*}^j\lambda+\mathbb{Z}^2\right).
\end{align*}
Taking an arbitrary point $v_1\in\Bbb Z^2$, we have
\begin{align*}
 L^{\varphi(p)j}M^{*^{-j}}({M^*}^j\lambda+v_1)= L^{\varphi(p)j}\lambda+L^{\varphi(p)j}M^{*^{-j}}v_1.
\end{align*}
By \eqref{eq(3.4)}, we have $L^{\varphi(p)j}=(np+1)^j=pm+1$ for some integer $m$.
It follows from $L^{\varphi(p)j}M^{*^{-j}}\in M_2(\mathbb{Z})$ and $pm\lambda\in\mathbb{Z}^2$ that
\begin{align*}
L^{\varphi(p)j}\lambda+L^{\varphi(p)j}M^{*^{-j}}v_1=\lambda+pm\lambda+L^{\varphi(p)j}M^{*^{-j}}v_1\in \lambda+\mathbb{Z}^2.
\end{align*}
This shows that \eqref{eq(3.5)} holds.  Hence, form \eqref{eq(2.3)}, \eqref{eq(1.6)} and \eqref{eq(3.5)}, we deduce that
\begin{align}\label{eq(3.6)} \nonumber
\mathcal{ Z}(\hat{\mu}_{M,D})&=\bigcup_{j=1}^{\infty}M^{*j}(\mathcal{ Z}(m_D))=\bigcup_{j=1}^{\infty}M^{*j}(\mathcal{Z}_{D}^2+\mathbb{Z}^2) \\
&\supset\bigcup_{j=1}^{\infty}L^{\varphi(p)j}(M^{*j}\mathcal{Z}_{D}^2+\mathbb{Z}^2)\supset\bigcup_{j=1}^{p^2-1}L^{\varphi(p)j}(M^{*j}\mathcal{Z}_{D}^2+\mathbb{Z}^2).
\end{align}

Let $\Lambda=L^{\varphi(p)(p^2-1)}E^2_p$. We will show that ${\mathcal E}_{\Lambda}=\{e^{2\pi i\langle\lambda,x\rangle}:\lambda\in\Lambda\}$
is an orthogonal set in $L^2(\mu_{M,D})$.  For any $\lambda_1\neq \lambda_2 \in \Lambda$, there exists $\lambda'\in \mathring{E}_p^2$ (mod $ \mathbb{Z}^2$)
such that $\lambda_1-\lambda_2=L^{\varphi(p)(p^2-1)}\lambda'$. Since $\bigcup_{j=1}^{p^2-1}{M^*}^j\mathcal{Z}_{D}^2=\mathring{E}_p^2$ (mod $ \mathbb{Z}^2$)
for $M\in \mathfrak{M}_p^{(2)}$, there exist $\lambda_0\in\mathcal{Z}_{D}^2$ and $j_0\in \{1,\cdots,p^2-1\}$ such that
$\lambda'={M^*}^{j_0}\lambda_0$ (mod $ \mathbb{Z}^2$). Then
\begin{align}\label{eq(3.7)}
\lambda_1-\lambda_2&\in L^{\varphi(p)(p^2-1)}\left({M^*}^{j_0}\lambda_0+\mathbb{Z}^2\right)=L^{\varphi(p)j_0}\left(L^{\varphi(p)(p^2-1-j_0)}
({M^*}^{j_0}\lambda_0+\mathbb{Z}^2)\right).
\end{align}
By using \eqref{eq(3.4)} again, we have
$L^{\varphi(p)(p^2-1-j_0)}=pm'+1$ for some integer $m'$. It
follows from $pm'{M^*}^{j_0}\lambda_0\in\mathbb{Z}^2$ that
$L^{\varphi(p)(p^2-1-j_0)}({M^*}^{j_0}\lambda_0+\mathbb{Z}^2)=(pm'+1)({M^*}^{j_0}\lambda_0+\mathbb{Z}^2)
\subset {M^*}^{j_0}\lambda_0+\mathbb{Z}^2$. Hence, by
\eqref{eq(3.6)} and  \eqref{eq(3.7)}, we have
\begin{align*}
\lambda_1-\lambda_2&\in  L^{\varphi(p)j_0}({M^*}^{j_0}\lambda_0+\mathbb{Z}^2)\subset L^{\varphi(p)j_0}({M^*}^{j_0}\mathcal{Z}_{D}^2+\mathbb{Z}^2)\\
&\subset\bigcup_{j=1}^{p^2-1} L^{\varphi(p)j}({M^*}^{j}\mathcal{Z}_{D}^2+\mathbb{Z}^2)\subset \mathcal{Z}(\hat{\mu}_{M,D}).
\end{align*}
This shows that $(\Lambda-\Lambda)\setminus\{0\}\subset
\mathcal{Z}(\hat{\mu}_{M,D})$. It follows from \eqref{eq(2.2)}
that the elements in $\mathcal {E}_\Lambda$ are mutually
orthogonal, and hence $n^*(\mu_{M,D})\geq p^2$.
\par
Combining the above with  \eqref{eq(3.3)}, we have $n^*(\mu_{M,D})=p^2$.
\end{proof}

Now we are ready to prove Theorem \ref{th(1.1)} by  Theorem \ref{th(1.3)} and  Proposition \ref{th(2.6)}.

\begin{proof}[The proof of Theorem \ref{th(1.1)}]
By Theorem \ref{th(1.3)}, there exist at most $p^2$ mutually
orthogonal exponential functions in $L^2(\mu_{M,D})$. We only need
to prove that the number $p^2$ is the best possible for $p$ is a
prime. Let $\widetilde{M}$ be the matrix satisfies  Proposition
\ref{th(2.6)}. Since $\emptyset\neq\mathcal{Z}_{D}^2\subset
\mathring{E}_p^2$, there exists at least one element $\lambda\in
\mathcal{Z}_{D}^2\subset \mathring{E}_p^2$, and then  Proposition
\ref{th(2.6)} shows that
$\bigcup_{j=1}^{p^2-1}\widetilde{M}^j\lambda =\mathring{E}_p^2$
(mod $\mathbb{Z}^2$). If $\widetilde{M}$ is an expanding matrix,
then $n^*(\mu_{\widetilde{M}^*,D})=p^2$ by Theorem \ref{th(1.3)}.
If not, choosing a sufficient large positive integer $N=pm+1$ such
that $\widehat{M}=N\widetilde{M}$ is expanding. Note that
$pm\widetilde{M}^j\lambda\in\mathbb{Z}^2$ for any $j$. It is easy
to see that $\bigcup_{j=1}^{p^2-1}\widehat{M}^j\lambda
=\bigcup_{j=1}^{p^2-1}\widetilde{M}^j\lambda =\mathring{E}_p^2$
(mod $\mathbb{Z}^2$), and then $n^*(\mu_{\widehat{M}^*,D})=p^2$ by
using Theorem \ref{th(1.3)} again.
\end{proof}

\begin{proof}[The proof of Theorem \ref{th(1.4)}]
For $p=2$ or $3$, the conclusion $n^*(\mu_{M,D})=p^2$ for $M\in \mathfrak{M}_p^{(2)}$  follows from Theorem \ref{th(1.3)}.
Hence to prove Theorem \ref{th(1.4)},
we only need to prove $n^*(\mu_{M,D})=p$ for $M\in \mathfrak{M}_p^{(1)}$.

Note that $\emptyset\neq\mathcal{Z}_{D}^2\subset
\mathring{E}_p^2$, there exists at least one element $\lambda\in
\mathcal{Z}_{D}^2\subset \mathring{E}_p^2$. It is easy to see that
$-\lambda\in\mathcal {Z}(m_{D})$ because
$m_{D}(-\lambda)=\frac{1}{\#(D)}\sum_{d\in D}{e^{2\pi i\langle d,
-\lambda\rangle}}=\frac{1}{\#(D)}\overline{\left(\sum_{d\in
D}{e^{2\pi i\langle d,\lambda\rangle}}\right)}=0$. By
$\mathbb{Z}^2$-periodic of $m_{D}$, we have
$(\lambda+\mathbb{Z}^2)\bigcup(-\lambda+\mathbb{Z}^2)\subset\mathcal
{Z}(m_{D})$. It follows from \eqref{eq(2.3)} that
\begin{align}\label{eq(3.8)}
 \mathcal{Z}(\hat{\mu}_{M,D})\supset \left(\bigcup_{j=1}^\infty{M^*}^j(\lambda+\mathbb{Z}^2)\right)\bigcup
 \left(\bigcup_{j=1}^\infty{M^*}^j(-\lambda+\mathbb{Z}^2)\right).
\end{align}

 (i) If $p=2$, then $n^*(\mu_{M,D})\leq \widetilde{P}=\frac{p^2}{2}=2$ by Theorem \ref{th(1.3)}.
 We now prove $n^*(\mu_{M,D})\geq2$. Let $\Lambda=\{0,s_1\}$  with $s_1=M^*\lambda$  and $\mathcal {E}_\Lambda=\{0, e^{2\pi i\langle s_1,x\rangle}\}$.
It is obvious that  $(\Lambda-\Lambda)\setminus\{0\}\subset \mathcal{Z}(\hat{\mu}_{M,D})$ since $\pm s_1\in\mathcal{Z}(\hat{\mu}_{M,D})$.
By \eqref{eq(2.2)}, $\mathcal {E}_\Lambda$ is an orthogonal exponential function set in $L^2(\mu_{M,D})$, and so $n^*(\mu_{M,D})\geq2$.
This completes the proof
of Theorem \ref{th(1.4)} for $p=2$.

 (ii) If $p=3$, then $n^*(\mu_{M,D})\leq \widetilde{P}=3$ by Theorem \ref{th(1.3)}.
Let $\Lambda=\{0,s_1,s_2\}$  with $s_1=M^*\lambda$ and $s_2=-M^*\lambda$, and let
\begin{align*}
\mathcal {E}_\Lambda=\{0, e^{2\pi i\langle s_1,x\rangle}, e^{2\pi i\langle s_2,x\rangle}\}.
\end{align*}
It is easy to see that $\pm s_1, \pm s_2\in
\mathcal{Z}(\hat{\mu}_{M,D})$. By \eqref{eq(3.8)} and $3\lambda\in\mathbb{Z}^2$, we have $s_1-s_2=M^*(2\lambda)=M^*(-\lambda+3\lambda)\in
M^*(-\lambda+\mathbb{Z}^2)\subset \mathcal{Z}(\hat{\mu}_{M,D})$.
Similarly, $s_2-s_1\in \mathcal{Z}(\hat{\mu}_{M,D})$. These show
that $(\Lambda-\Lambda)\setminus\{0\}\subset
\mathcal{Z}(\hat{\mu}_{M,D})$. By \eqref{eq(2.2)}, $\mathcal
{E}_\Lambda$ is an orthogonal exponential function set in
$L^2(\mu_{M,D})$, hence $n^*(\mu_{M,D})\geq3$.

This completes the proof of  Theorem \ref{th(1.4)}.
\end{proof}

\section{\bf The Sierpinski measures and an example \label{sect.4}}

As an application of the above-mentioned results, in this section
we first complete the discussion of planar Sierpinski measures in
\cite{Li_2009}, then we construct an example in which
$\mathfrak{M}_p^{(1)}=\emptyset$.

As in \cite{Li_2009}, by using the residue system of modulo $3$, the transposed conjugate $M^*$ of $M=\left[ {\begin{array}{*{20}{c}}
{{a}}&{{b}}\\
{{d}}&{{{c}}}
\end{array}} \right]\in M_2(\mathbb{Z})$ can be expressed in the following way:
\begin{equation}\label{eq(4.1)}
M^*=\left[ {\begin{array}{*{20}{c}}
{a}&{d}\\
{b}&{c}
\end{array}} \right]=3\left[ {\begin{array}{*{20}{c}}
{l_1}&{l_2}\\
{l_3}&{l_4}
\end{array}} \right]+M_\alpha=3\widetilde{M}+M_\alpha,
\end{equation}
where $\widetilde{M}\in M_2(\mathbb{Z})$ and the entries of the
matrix $M_\alpha$ are $0, 1$ or $2$. It is easy to see that
$M^*\lambda=M_\alpha\lambda$ (mod $\mathbb{Z}^2$) for $\lambda\in
E_3^2\setminus\{0\}$. This shows that if we want to check the
matrix $M$ belongs to $\mathfrak{M}_3^{(1)}$ or $\mathfrak{M}_3^{(1)}$, we
only need to check  $M_\alpha$.  We can easily prove  that
\begin{align*}
{\rm there\; exist\; only}\;48\;{\rm different}\;M_{\alpha}\;{\rm denoted\; by}\;\{M_{\alpha}\}_{\alpha=1}^{48}.
\end{align*}
Obviously, for each fixed $M_{\alpha}$, there are infinitely many
expanding matrix $M\in M_2(\Bbb Z)$ with $\det(M)\not\in 3\Bbb Z$
such that $M^*-M_{\alpha}\in M(3\Bbb Z)$.  By using the results in
\cite{Li_2009}, and rearranging the indexes of $M_{\alpha}$ if
necessary, we give the concrete expression of matrix $M_{\alpha}$
as follows:
\par
\begin{equation*}
\begin{array}{l}
{M_1} = \left[ {\begin{array}{*{20}{c}}
0&1\\
1&0
\end{array}} \right], \;  \  {M_2} = \left[ {\begin{array}{*{20}{c}}
0&2\\
1&2
\end{array}} \right], \;  \  {M_3} = \left[ {\begin{array}{*{20}{c}}
0&2\\
2&0
\end{array}} \right], \;   \  {M_4} = \left[ {\begin{array}{*{20}{c}}
0&1\\
2&1
\end{array}} \right],\\
{M_5} = \left[ {\begin{array}{*{20}{c}}
1&0\\
0&1
\end{array}} \right], \ \   {M_6} = \left[ {\begin{array}{*{20}{c}}
1&2\\
0&2
\end{array}} \right], \ \  \   {M_7} = \left[ {\begin{array}{*{20}{c}}
1&2\\
1&0
\end{array}} \right], \ \ {M_8} = \left[ {\begin{array}{*{20}{c}}
1&0\\
1&2
\end{array}} \right],\\
{M_9} = \left[ {\begin{array}{*{20}{c}}
2&1\\
0&1
\end{array}} \right],\;  {M_{10}} = \left[ {\begin{array}{*{20}{c}}
2&0\\
0&2
\end{array}} \right], \;  {M_{11}} = \left[ {\begin{array}{*{20}{c}}
2&1\\
2&0
\end{array}} \right], \;  {M_{12}} = \left[ {\begin{array}{*{20}{c}}
2&0\\
2&1
\end{array}} \right]\\
{M_{13}} = \left[ {\begin{array}{*{20}{c}}
0&2\\
1&0
\end{array}} \right], \;    {M_{14}} = \left[ {\begin{array}{*{20}{c}}
0&1\\
2&0
\end{array}} \right], \;    {M_{15}} = \left[ {\begin{array}{*{20}{c}}
1&0\\
0&2
\end{array}} \right], \;     {M_{16}} = \left[ {\begin{array}{*{20}{c}}
1&1\\
0&2
\end{array}} \right],\\
{M_{17}} = \left[ {\begin{array}{*{20}{c}}
1&1\\
1&2
\end{array}} \right], \;  {M_{18}} = \left[ {\begin{array}{*{20}{c}}
1&0\\
2&2
\end{array}} \right], \;   {M_{19}} = \left[ {\begin{array}{*{20}{c}}
1&2\\
2&2
\end{array}} \right], \; {M_{20}} = \left[ {\begin{array}{*{20}{c}}
2&0\\
0&1
\end{array}} \right],\\
{M_{21}} = \left[ {\begin{array}{*{20}{c}}
2&2\\
0&1
\end{array}} \right],\;  {M_{22}} = \left[ {\begin{array}{*{20}{c}}
2&0\\
1&1
\end{array}} \right], \;  {M_{23}} = \left[ {\begin{array}{*{20}{c}}
2&1\\
1&1
\end{array}} \right], \;  {M_{24}} = \left[ {\begin{array}{*{20}{c}}
2&2\\
2&1
\end{array}} \right].
\end{array}
\end{equation*}
\begin{equation*}
\begin{array}{l}
{M_{25}} = \left[ {\begin{array}{*{20}{c}}
0&2\\
1&1
\end{array}} \right], \;    {M_{26}} = \left[ {\begin{array}{*{20}{c}}
0&1\\
2&2
\end{array}} \right], \;    {M_{27}} = \left[ {\begin{array}{*{20}{c}}
1&1\\
0&1
\end{array}} \right], \;     {M_{28}} = \left[ {\begin{array}{*{20}{c}}
1&2\\
0&1
\end{array}} \right],\\
{M_{29}} = \left[ {\begin{array}{*{20}{c}}
1&0\\
1&1
\end{array}} \right], \;  {M_{30}} = \left[ {\begin{array}{*{20}{c}}
1&1\\
2&0
\end{array}} \right], \;   {M_{31}} = \left[ {\begin{array}{*{20}{c}}
1&0\\
2&1
\end{array}} \right], \; {M_{32}} = \left[ {\begin{array}{*{20}{c}}
2&1\\
0&2
\end{array}} \right],\\
{M_{33}} = \left[ {\begin{array}{*{20}{c}}
2&2\\
0&2
\end{array}} \right],\;  {M_{34}} = \left[ {\begin{array}{*{20}{c}}
2&2\\
1&0
\end{array}} \right], \;  {M_{35}} = \left[ {\begin{array}{*{20}{c}}
2&0\\
1&2
\end{array}} \right], \;  {M_{36}} = \left[ {\begin{array}{*{20}{c}}
2&0\\
2&2
\end{array}} \right].
\end{array}
\end{equation*}
\begin{equation*}
\begin{array}{l}
{M_{37}} = \left[ {\begin{array}{*{20}{c}}
0&1\\
1&1
\end{array}} \right], \;    {M_{38}} = \left[ {\begin{array}{*{20}{c}}
0&1\\
1&2
\end{array}} \right], \;    {M_{39}} = \left[ {\begin{array}{*{20}{c}}
0&2\\
2&1
\end{array}} \right], \;     {M_{40}} = \left[ {\begin{array}{*{20}{c}}
0&2\\
2&2
\end{array}} \right],\\
{M_{41}} = \left[ {\begin{array}{*{20}{c}}
1&1\\
1&0
\end{array}} \right], \;  {M_{42}} = \left[ {\begin{array}{*{20}{c}}
1&2\\
1&1
\end{array}} \right], \;   {M_{43}} = \left[ {\begin{array}{*{20}{c}}
1&2\\
2&0
\end{array}} \right], \; {M_{44}} = \left[ {\begin{array}{*{20}{c}}
1&1\\
2&1
\end{array}} \right],\\
{M_{45}} = \left[ {\begin{array}{*{20}{c}}
2&1\\
1&0
\end{array}} \right],\;  {M_{46}} = \left[ {\begin{array}{*{20}{c}}
2&2\\
1&2
\end{array}} \right], \;  {M_{47}} = \left[ {\begin{array}{*{20}{c}}
2&2\\
2&0
\end{array}} \right], \;  {M_{48}} = \left[ {\begin{array}{*{20}{c}}
2&1\\
2&2
\end{array}} \right].
\end{array}
\end{equation*}

For
\begin{align}\label{eq(4.2)}
D=\left\{ {\left( {\begin{array}{*{20}{c}}
0\\
0
\end{array}} \right),\left( {\begin{array}{*{20}{c}}
1\\
0
\end{array}} \right),\left( {\begin{array}{*{20}{c}}
0\\
1
\end{array}} \right)} \right\},
\end{align}
it is easy to show that
$\mathcal{Z}_{D}^2=\{(1/3,2/3)^t,(2/3,1/3)^t \}$. By a direct
calculation or by an application of Propositions $3-6$ in
\cite{Li_2009}, we have $M_\alpha\in\mathfrak{M}_3^{(1)}$ if
$\alpha\in \{1,\cdots,36\}$ and $M_\alpha\in\mathfrak{M}_3^{(2)}$
if $\alpha\in \{37,\cdots,48\}$. Hence the following corollary can
be derived from Theorem \ref{th(1.4)} directly.

\begin{cor} \label{th(4.1)}
Suppose that the self-affine measure $\mu_{M, D}$ is defined by
\eqref{eq(1.1)}, where $D$ is given by \eqref{eq(4.2)}, and $M$ is
an expanding integer matrix with $\det(M)\notin 3\mathbb{Z}$. Let
$n^*(\mu_{M,D})$ be given by \eqref{eq(1.7)} and
$M^*=3\widetilde{M}+M_\alpha$ as \eqref{eq(4.1)}, then
\begin{align*}
n^*(\mu_{M, D})= \left\{
\begin{array}{ll}
3,\quad &{\rm if}\;  \alpha\in \{1,\cdots,36\}, \\
9,\quad &{\rm if}\;   \alpha\in \{37,\cdots,48\} .
\end{array}
\right.
\end{align*}
\end{cor}

\medskip

\begin{re}
In \cite{Li_2009}, the  proof contains a gap and the conclusion is
not correct for $\mathfrak{M}_3^{(2)}$. In fact, the third author
of this paper assumes that there exist $4$  exponential functions
$\{ e^{2\pi i\langle\lambda_j,x\rangle}:j=1,2,3,4\}$ being
mutually orthogonal in  $L^2(\mu_{M, D})$. Then the differences
$\lambda_i-\lambda_j\in \mathcal{Z}(\widehat{\mu}_{M,
D})=\bigcup_{j=1}^4(Z_j\bigcup\widetilde{Z}_j)$ for $i\not=j$(see
Proposition 2.6 in \cite{Li_2009}), and thus the following six
differences
$$
\lambda_1-\lambda_2, \  \  \lambda_1-\lambda_3, \  \  \lambda_1-\lambda_4, \ \  \lambda_2-\lambda_3, \ \  \lambda_2-\lambda_4, \ \  \lambda_3-\lambda_4
$$
belong to $\mathcal{Z}(\widehat{\mu}_{M, D})$. One can regard the
above eight sets $Z_1, Z_2, Z_3, Z_4, \widetilde{Z}_1,
\widetilde{Z}_2, \widetilde{Z}_3, \widetilde{Z}_4$ as eight small
boxes and let the above six differences belong to these eight
small boxes. It has finite many possible distributions, and the
following distribution is not considered in \cite{Li_2009}:
$$
\begin{array}{|c|c|c|c|c|c|c|c|}
\hline
Z_1&Z_2&Z_3&Z_4&\tilde{Z_1}&\tilde{Z_2}&\tilde{Z_3}&\tilde{Z_4}\\
\hline
 \lambda_1-\lambda_2& \lambda_1-\lambda_3 &\lambda_1-\lambda_4 &\lambda_3-\lambda_2 & \lambda_2-\lambda_1
 & \lambda_3-\lambda_1 & \lambda_4-\lambda_1 & \lambda_2-\lambda_3  \\
 \lambda_3-\lambda_4& \lambda_2-\lambda_4 & & & \lambda_4-\lambda_3 & \lambda_4-\lambda_2  &  & \\
\hline
\end{array}
$$
For the above distribution, we can not get a contradiction by
Propositions $2$ and $6$ in \cite{Li_2009}. However, the above
Corollary 4.1 gives a complete result.
\end{re}

In the end of this paper, we construct an example to illustrate
the case that $\mathfrak{M}_p^{(1)}=\emptyset$.

\begin{ex}\label{th(4.2)}
Let
\begin{align*}
D_1=\left\{ {\left( {\begin{array}{*{20}{c}}
0\\
0
\end{array}} \right),\left( {\begin{array}{*{20}{c}}
-1\\
0
\end{array}} \right),\left( {\begin{array}{*{20}{c}}
1\\
1
\end{array}} \right)} \right\}, \ \
D_2=\left\{ {\left( {\begin{array}{*{20}{c}}
0\\
0
\end{array}} \right),\left( {\begin{array}{*{20}{c}}
3\\
1
\end{array}} \right),\left( {\begin{array}{*{20}{c}}
0\\
-1
\end{array}} \right)} \right\}
\end{align*}
and $D_3=D_1 + D_2$. Then $n^*(\mu_{M,D_{3}})=9$ for any expanding
integer matrix $M$ with  $\det(M)\notin 3\mathbb{Z}$, i,e.,
$\mathfrak{M}_3^{(1)}=\emptyset$.
\end{ex}
\begin{proof}
Since $D_3=D_1 + D_2$, we have $m_{D_3}(x)=m_{D_1}(x)m_{D_2}(x)$ and  $\mathcal {Z}(m_{D_3}(x))=\mathcal {Z}(m_{D_1}(x))
\cup \mathcal {Z}(m_{D_2}(x))$.
Let $x=(x_1,x_2)^t$, then $m_{D_2}(x)=1+e^{2\pi (3x_1+x_2)}+e^{2\pi(-x_2)}$. It is well known that $m_{D_2}(x)=0$ only if
\begin{align*}
\left\{
\begin{array}{lll}
3x_1+x_2&=&\frac{1}{3}+k_1,\\
-x_2&=&\frac{2}{3}+k_2,\\
\end{array}\right.    \ \ \ {\rm{or}} \ \ \ \
\left\{
\begin{array}{lll}
3x_1+x_2&=&\frac{2}{3}+k_1,\\
-x_2&=&\frac{1}{3}+k_2,\\
\end{array}\right.
\end{align*}
where $k_1$ and $k_2$ are integers.
It is easy to show that
$$
\mathcal{Z}_{D_2}^2=\left\{ {\left( {\begin{array}{*{20}{c}}
0\\
1/3
\end{array}} \right),\left( {\begin{array}{*{20}{c}}
1/3\\
1/3
\end{array}} \right),\left( {\begin{array}{*{20}{c}}
2/3\\
1/3
\end{array}} \right),\left( {\begin{array}{*{20}{c}}
0\\
2/3
\end{array}} \right),\left( {\begin{array}{*{20}{c}}
1/3\\
2/3
\end{array}} \right),\left( {\begin{array}{*{20}{c}}
2/3\\
2/3
\end{array}} \right)} \right\}.
$$
Similarly,
$$
\mathcal{Z}_{D_1}^2=\left\{ {\left( {\begin{array}{*{20}{c}}
1/3\\
0
\end{array}} \right),\left( {\begin{array}{*{20}{c}}
2/3\\
0
\end{array}} \right)} \right\}.
$$
Hence $\mathcal{Z}_{D_3}^2=\mathcal{Z}_{D_1}^2\bigcup\mathcal{Z}_{D_2}^2=\mathring{E}_3^2.$

By Proposition \ref{th(2.2)}, we have $M^*(\mathcal{Z}_{D_3}^2)=M^*(\mathring{E}_3^2)=\mathring{E}_3^2$(mod $\mathbb{Z}^2$)
for any expanding matrix $M$ with $\gcd(\det(M),3)=1$, and then
$\bigcup_{j=1}^{8}{M^*}^j\mathcal{Z}_{D_3}^2=\mathring{E}_3^2$ (mod $ \mathbb{Z}^2$). Hence $M\in\mathfrak{M}_3^{(2)}$ and $n^*(\mu_{M,D_3})=9$.
\end{proof}

\end{document}